\documentclass[11pt]{amsart}
\usepackage{geometry}                
\usepackage{graphicx}
\usepackage{amssymb}
\usepackage{amscd}
\usepackage{epstopdf}
\title{Almost-K\"ahler Anti-Self-Dual metrics}

\begin{document}
\author{Inyoung Kim}
\address{The Center for Geometry and its Applications, 
Pohang University of Science and Technology,
Pohang city, 790-784, South Korea}
\email{kiysd@postech.ac.kr}

\thanks{The research of the author was supported in part by SRC-GAIA, the Grant 2011-0030044 from The Ministry of Education, The Republic of Korea.}
\maketitle

  \begin{abstract}
We show the existence of strictly almost-K\"ahler anti-self-dual metrics on certain 4-manifolds by deforming scalar-flat K\"ahler metrics.
  On the other hand, we prove the non-existence of such metrics on certain other 4-manifolds by means of Seiberg-Witten theory. In the process, 
we provide a simple new proof of the fact that any almost-K\"ahler anti-self-dual 4-manifold must have a non-trivial Seiberg-Witten invariant.

\end{abstract}
\maketitle
\section{\large\textbf{Introduction}}\label{S:Intro}
On a smooth, oriented riemannian 4-manifold, the bundle of 2-forms $\Lambda^{2}$ decomposes as self-dual and anti-self-dual 2-forms 
\begin{equation}
\Lambda^{2}=\Lambda^{+}\oplus \Lambda^{-},
\end{equation}
where $\Lambda^{\pm}$ are $(\pm 1)$- eigenspaces of the Hodge-star operator.
In terms of this decomposition,  the riemannian curvature operator $R:\Lambda^{2}\to\Lambda^{2}$ takes the form,

\[
 R=
\begin{pmatrix}
 \begin{array}{c|c}
 W_{+}+\frac{s}{12}& \mathring{r}\\
 \hline
  \mathring{r}&W_{-}+\frac{s}{12}
 \end{array}
 \end{pmatrix}.
 \]
If $W_{+}=0$, the the riemannian metric is said anti-self-dual (ASD). 
Anti-self-dual metrics are hard to construct, but most generally, Taubes showed that 
$M\#n\overline{{\mathbb{CP}}^{2}}$ admits such a metric for all sufficiently large $n$, where $M$ is any smooth compact, oriented 4-manifold [26].

On the other hand, when a 4-manifold admits an additional structure such as a complex structure or a symplectic structure, it is natural to think of compatible metrics.
More precisely, let $(M,\omega)$ be a symplectic 4-manifold.
The symplectic form  $\omega$ is a closed and nondegenerate 2-form. 
By this, we mean $d\omega=0$, and for each $x\in M$ and nonzero $v\in T_{x}M$,  there exists $w\in T_{x}M$ such that 
$\omega(v,w)\neq 0$. 
A smooth fiber-wise linear map $J: TM\to TM$ on $M$ is called an almost-complex structure when $J$ satisfies $J^{2}=-1$.
We say $\omega$ is compatible with $J$ if 
\[\omega(Jv_{1}, Jv_{2})=\omega(v_{1}, v_{2})\]
for each $x\in M$ and $v_{1}, v_{2}\in T_{x}M$ and $\omega(v, Jv)>0$ for all nonzero $v\in T_{x}M$. 
It is known that the space of almost-complex structures which are compatible with $\omega$ is nonempty and contractible [19]. 
Given such a compatible $J$, we define the associated metric by
\[g(v,w)=\omega(v,Jw).\]
Then $g$ is a positive symmetric bilinear form and $J$ is compatible with $g$, that is, $g(v,w)=g(Jv, Jw)$. 
When such a $g$ is anti-self-dual, we call $g$ an almost-K\"ahler anti-self-dual metric.

When $J$ is integrable,  $(g,\omega, J)$ is said to be a K\"ahler structure. 
A K\"ahler metric on complex surfaces is scalar-flat if and only if it is anti-self-dual [17].
Thus, scalar-flat K\"ahler metrics provides one of the important examples of ASD metrics.  

The main topic of this article is to study an almost-K\"ahler anti-self-dual metric on a smooth, compact 4-manifold. 
We call $g$ a strictly almost-K\"ahler ASD metric when it is not scalar-flat K\"ahler. 
We show the existence of strictly almost-K\"ahler ASD metrics by deforming scalar-flat K\"ahler metrics. 
Many examples of scalar-flat K\"ahler metrics have been constructed by Kim, LeBrun, Pontecorvo [10], [13] and Rollin, Singer [24].
The deformation theory of such metrics has been studied [17]. 
Fortunately, twistor theory provides a transparent interpretation of deformation of ASD metrics in terms of Kodaira-Spencer theory of deformations of complex structure. 

Conversely, we will give a complete classification (up to diffeomorphism) of the smooth, compact oriented 4-manifolds 
which admit both almost-K\"ahler anti-self-dual metrics and metrics of positive scalar curvature. 
The main tool is the Seiberg-Witten invariant. 
In particular, we show that an almost-K\"ahler anti-self-dual 4-manifold has a unique solution of the Seiberg-Witten equation 
for an explicit perturbation form. 
Taubes already showed that the existence of the solution of the Seiberg-Witten equation for an almost-K\"ahler metric with large perturbation form [27]. 
We show that this equation becomes more explicit when we use an almost-K\"ahler anti-self-dual metric. 
In particular, we get an explicit bound of the perturbation form. 
Combining with Liu's theorem [18], we get useful topological information when we also assume
$M$ admits a positive scalar curvature metric.
In particular, we show that $\mathbb{CP}^{2}\#n\overline{\mathbb{CP}^{2}}$ admits an almost-K\"ahler anti-self-dual metric if and only if  $n\geq 10$.
We state our main results below.

\newtheorem{Theorem}{Theorem}
\begin{Theorem}
Suppose a smooth, compact 4-manifold $M$ admits an almost K\"ahler ASD metric. 
If $M$ also admits a metric of positive scalar curvature, then it is diffeomorphic to one of the following:

\begin{itemize}
   \item$\mathbb{CP}^{2}\#n\overline{\mathbb{CP}^{2}}$ for $n\geq 10$;
   
   \item$S^{2}\times \Sigma_{\mathbf{g}}$ and non-trivial $S^{2}$-bundle over $\Sigma_{\mathbf{g}}$, where $\Sigma_{\mathbf{g}}$ is a Riemann surface with 
   genus $\mathbf{g}\geq 2$;
   \item$(S^{2}\times \Sigma_{\mathbf{g}})\#n\overline{\mathbb{CP}^{2}}$ for $n\geq 1$; or
   \item$(S^{2}\times T^{2})\#n\overline{\mathbb{CP}^{2}}$ for $n\geq 1$.
 \end{itemize}
Conversely, each of the following differentiable manifolds admits an almost-K\"ahler anti-self-dual metric:

\begin{itemize}
  \item$\mathbb{CP}^{2}\#n\overline{\mathbb{CP}^{2}}$  for $n\geq 10$;
  \item$S^{2}\times \Sigma_{\mathbf{g}}$ and non-trivial $S^{2}$-bundle over $\Sigma_{\mathbf{g}}$, where $\Sigma_{\mathbf{g}}$ is a Riemann surface with genus $\mathbf{g}\geq 2$;
  \item$(S^{2}\times \Sigma_{\mathbf{g}})\#n\overline{\mathbb{CP}^{2}}$ for $n\geq 1$; or
  \item$(S^{2}\times T^{2})\#n\overline{\mathbb{CP}^{2}}$ for $n\geq 4$.
 \end{itemize}
 Except $\mathbb{CP}^{2}\#10\overline{\mathbb{CP}^{2}}$ and $S^{2}\times \Sigma_{\mathbf{g}}$, each of those admits a strictly almost-K\"ahler anti-self-dual metric. 
\end{Theorem}
\noindent
\textbf{Acknowledgement}. This article is based on the author's Ph.D. Thesis. 
The author would like to thank Prof. Claude LeBrun for suggesting the problem,
as well as for his constant help and encouragement.

\section{\large\textbf{Geometry of almost-K\"ahler anti-self-dual metrics}}\label{S:P}
We begin by discussing the basic property of almost-K\"ahler anti-self-dual metrics. 
Throughout this section, we assume $M$ is a smooth, compact, oriented 4-manifold.
\newtheorem{Lemma}{Lemma} 
\begin{Lemma}
Let $(M, g, \omega, J)$ be an almost-K\"ahler structure. 
Then $\omega$ is a self-dual harmonic 2-form $\omega$ with $|\omega|=\sqrt{2}$. 
\end{Lemma}
\begin{proof}The easiest way to see this is to use an orthonormal basis. 
Since $J$ is orthogonal with respect to $g$, there exists an orthonormal basis of the form $\{e_{1}, e_{2}=J(e_{1}), e_{3}, e_{4}=J(e_{3})\}$.
Then the corresponding 2-form $\omega$ is 
\[e_{1}\wedge e_{2} + e_{3}\wedge e_{4}.\]
This is self-dual and $|\omega|=\sqrt{2}$.
Since $\omega$ is closed and self-dual,  we get 
\[d^{*}\omega=-\ast d\ast\omega=-\ast d\omega=0,\]
and therefore, 
\[\Delta\omega=(dd^{*}+d^{*}d)\omega=0.\]
\end{proof}
\begin{Lemma}
The scalar curvature of almost-K\"ahler anti-self-dual metric is nonpositive. 
Moreover, the scalar curvature is identically zero if and only if the metric is K\"ahler.
\end{Lemma}
\begin{proof}
On an oriented, smooth, compact  riemannian 4-manifold, there is well-known Weitzenb\"ock formula [3] for self-dual 2-forms
\begin{equation}
\Delta\omega=\nabla^{*}\nabla\omega-2W_{+}(\omega, \cdot)+\frac{s}{3}\omega.
\end{equation}

If we take an inner product with $\omega$ in $(2)$, we get
\[<\Delta\omega,\omega>=<\nabla^{*}\nabla\omega,\omega>-2W_{+}(\omega, \omega)+\frac{s}{3}|\omega|^{2}.\]
If $g$ is an almost-K\"ahler, then the corresponding symplectic 2-form $\omega$ is harmonic and $|\omega|=\sqrt{2}$. Thus, we get
\[0=\frac{1}{2}\Delta|\omega|^{2}+|\nabla\omega|^{2}-2W_{+}(\omega,\omega)+\frac{s}{3}|\omega|^{2}\]
\[0=|\nabla\omega|^{2}-2W_{+}(\omega,\omega)+\frac{2s}{3}\]
Since $g$ is anti-self-dual, we get 
\begin{equation}
0=|\nabla\omega|^{2}+\frac{2s}{3}.
\end{equation}
\end{proof}

Note that an anti-self-dual metric is real analytic in suitable coordinates. 
Thus, $\nabla\omega$ is also real analytic and therefore, it does not vanish on an open set. 
Moreover, it can only vanishes on the union of real analytic subvarieties.
Thus, if $g$ is a strictly almost-K\"ahler anti-self-dual metric, then $s\neq 0$ on an open dense subset. 
 
 When there is an almost-complex structure $J$ on $M$, the complexified tangent vector bundle decomposes as
\[TM\otimes\mathbb{C}=T^{1,0}\oplus T^{0,1},\] 
where 
\[T^{1,0}=\{Z=X-iJ(X)\in T^{\mathbb{C}} | X\in TM\}\]
\[T^{0,1}=\{Z=X+iJ(X)\in T^{\mathbb{C}} | X\in TM\}.\]

When we evaluate 2-forms on complexified vectors, we extend it complex linearly on both factors. 
Accordingly, we define extended metric on $T^{\mathbb{C}}$ by
\[g_{\mathbb{C}}(iX, Y)=ig_{\mathbb{C}}(X, Y)\]
\[g_{\mathbb{C}}(X, iY)=ig_{\mathbb{C}}(X, Y),\]
where $X, Y \in TM$. 
In this convention, we have 
\[g_{\mathbb{C}}(T^{1,0}, T^{1,0})=0.\]
Recall $\omega$ is related to $g$ by $\omega(X, Y)=g(JX, Y)$. This can be extended complex bilinearly on both factors and therefore, we can think of $\omega$ as an element of $\Lambda^{2}_{\mathbb{C}}$.
Then $\omega$ associated with metric $g_{\mathbb{C}}$ is an element of $\Lambda^{1,1}$.

Following the outline given in [25], we show that when $d\omega=0$, the zero set of Nijenhuis tensor $N$ is equal to the zero set of $\nabla\omega$. 
\newtheorem{Proposition}{Proposition}
\begin{Proposition}
Let $(M, g, J, \omega)$ be an almost-K\"ahler structure. 
Then the zero set of the Nijenhuis tensor $N$ is equal to the zero set of $\nabla\omega$. 
\end{Proposition}
\begin{proof}
We can check that 
\[(\nabla_{X}\omega)(Y, Z)=2ig(\nabla_{X}Y, Z),\]
for $Y, Z\in T^{1,0}$ and $X\in T^{\mathbb{C}}$.
And for $Y, Z \in T^{0,1}$, we have 
\[(\nabla_{X}\omega)(Y, Z)=-2ig(\nabla_{X}Y, Z).\]
Using this fact and  the torsion-free property of the Levi-Civita connection, we have 
\[g([X,Y],Z)=g(\nabla_{X}Y,Z)-g(\nabla_{Y}X, Z)\]
\[=-\frac{1}{2i}(\nabla_{X}\omega)(Y,Z)+\frac{1}{2i}(\nabla_{Y}\omega)(X, Z),\]
for $X, Y, Z\in T^{0,1}$.
If the Nijenhuis tensor vanishes, then $[X, Y]\in T^{0,1}$ and therefore, we have $g([X, Y], Z)=0$.
Thus, we get 
\[(\nabla_{X}\omega)(Y,Z)-(\nabla_{Y}\omega)(X, Z)=0\]
On the other hand, by definition we have
\[d\omega(X, Y, Z)=(\nabla_{X}\omega)(Y, Z)-(\nabla_{Y}\omega)(X, Z)+(\nabla_{Z}\omega)(X, Y).\]
Thus, if $d\omega=0$ and $N$ vanishes, then 
\[(\nabla_{Z}\omega)(X, Y)=0\]
for $X, Y, Z \in T^{0,1}$.
Also we have $(\nabla_{X}\omega)(Y, Z)=0$ and $(\nabla_{Y}\omega)(X, Z)=0$ for $X, Y\in T^{0,1}$ and $Z\in T^{1,0}$.
Then from $d\omega=0$, we get
\[(\nabla_{Z}\omega)(X, Y)=0,\]
for  $X, Y\in T^{0,1}$ and $Z\in T^{1,0}$.
Combining these two, we get $(\nabla_{Z}\omega)(X, Y)=0$ for $Z\in T^{\mathbb{C}}$ and $X, Y\in T^{0,1}$.

When $N=0$, we also have $[T^{1,0}, T^{1,0}]\in T^{1,0}$ and therefore $g([X, Y], Z)=0$ holds for $X, Y, Z\in T^{1,0}$. 
In the same way, we can show $(\nabla_{Z}\omega)(X,Y)=0$  for $Z\in T^{\mathbb{C}}$ and $X, Y\in T^{1,0}$.
Since $(\nabla_{X}\omega)(Y, Z)=0$ for $X\in T^{\mathbb{C}}$ and $Y\in T^{1,0}$,  $Z\in T^{0,1}$,  we can conclude when $d\omega=0$ holds, the zero set of $N$ is equal to the zero set of $\nabla\omega$.
\end{proof}

Armstrong found certain topological obstruction when the Nijenhuis tensor is nowhere vanishing [1].

\begin{Theorem}
(Armstrong) Let $(M, J)$ be a compact smooth 4-manifold with an almost-complex structure $J$.
 If the Nijenhuis tensor $N$ is nowhere vanishing, then $5\chi+6\tau=0$.
\end{Theorem}

\newtheorem{Corollary}{Corollary}
\begin{Corollary}
Suppose $M$ admits an almost-K\"ahler ASD metric $g$. If $5\chi+6\tau\neq 0$, then the scalar curvature $s$ vanishes somewhere. 
\end{Corollary}
\begin{proof}
By Armstrong's theorem, the Nijenhuis tensor should vanish somewhere. 
Since $d\omega=0$, $\nabla\omega=0$ vanish somewhere by Proposition 1. 
Since the zero set of the scalar curvature is equal to the zero set of $\nabla\omega$,
$s$ vanish somewhere and if $g$ is not K\"ahler, $s$ is negative somewhere. 
\end{proof}

\newtheorem{Example}{Example}
\begin{Example}

In the next section, we show that the following examples admit almost-K\"ahler anti-self-dual metrics for certain values $n$. 

$1)$ $ \mathbb{CP}^{2}\#n\overline{\mathbb{CP}^{2}}$.

In this case, $\chi=3+n$ and $\tau=1-n$. Using this, we have
\[5\chi+6\tau=5(3+n)+6(1-n)=21-n\]
Thus, except $n=21$, the scalar curvature of an almost-K\"ahler ASD metric should vanish somewhere. 

$2)$ $S^{2}\times T^{2}\#n\overline{\mathbb{CP}^{2}}$.

In this case, we have $\chi=n$ and $\tau=-n$. 
Thus, we have 
\[5\chi+6\tau=5n-6n=-n\]
Thus, for $n\geq 1$, the scalar curvature should vanish somewhere.

$3)$ $S^{2}\times \Sigma_{\mathbf{g}}\#n\overline{\mathbb{CP}^{2}}$.

In this case, we have $\chi=-4(\mathbf{g}-1)+n$ and $\tau=-n$. 
Thus, we have 
\[5\chi+6\tau=-20\mathbf{g}+5n+20-6n=-20\mathbf{g}-n+20\]
Thus $5\chi+6\tau$ is always negative for $n\geq 0$ and $\mathbf{g}\geq 2$. 

\end{Example}
 
 Suppose a smooth 4-manifold $M$ admits an almost-complex structure $J$ and a compatible metric $g$.
 Then the complex-valued 2-forms decompose as 
\[\Lambda^{2}T^{*}_{\mathbb{C}}=(\Lambda^{2.0}\oplus \Lambda^{0,2})\oplus\Lambda ^{1,1}.\]
On the other hand, on an oriented, smooth 4-manifold, we have the following decomposition
\[\Lambda^{2}T^{*}_{\mathbb{C}}=\Lambda^{+}_{\mathbb{C}}\oplus \Lambda^{-}_{\mathbb{C}}.\]
As before, we define an associated 2-form $\omega$ by $\omega(v, w)=g(Jv, w)$.
When it is not required $d\omega=0$, $(M, g, \omega, J)$ is called an almost-Hermitian manifold. 
In this case, two decompositions are compatible in the following way. 
\begin{Lemma}
Let $(M,g,\omega, J)$ be an almost-Hermitian 4-manifold. 
Then we have
\[\Lambda^{+}_{\mathbb{C}}=\mathbb{C}\omega\oplus \Lambda^{2.0}\oplus \Lambda^{0,2},\]
\[\Lambda^{-}_{\mathbb{C}}=\Lambda_{0}^{1,1},\]
where $\Lambda_{0}^{1,1}$ is the orthogonal complement of $\omega$ in the space of $\Lambda^{1,1}$.
\end{Lemma}

\section{\textbf{Deformation of scalar-flat K\"ahler metrics}}\label{S:P}
In this section, we show the existence of strictly almost-K\"ahler anti-self-dual metrics by deforming the scalar-flat K\'ahler metrics. 
Many examples of scalar-flat K\"ahler metrics are known [10], [13], [24] and their deformation theory has been studied [17]. 
Deformation theory of ASD metrics has been known in a purely differential geometric point of a view on a smooth, 
compact, oriented 4-manifold $M$ using the Atiyah-Singer Index theorem.
 We study them together at the twistor level using short exact sequences of sheaves.
This part mainly depends on [17]. The important well-known result in this context is the following [17].

\begin{Proposition}
 Let $g$ be a K\"ahler metric on a complex surface. 
Then $g$ is anti-self-dual if and only if its scalar curvature $s=0$.
\end{Proposition}

Let us consider the geometry of scalar-flat K\"ahler metrics briefly. 
This will be useful later when they are compared with almost-K\"ahler ASD metrics.
We define the Ricci form $\rho$ by
\[\rho(X,Y)=Ric(JX, Y).\]
 Then $\rho$ is a closed, real $(1,1)$ form.
The important theorem on a K\"ahler manifold is that  the first Chern class of the manifold is represented by $\frac{1}{2\pi}\rho$ 
\[c_{1}(M):=c_{1}(K^{-1})=[\frac{\rho}{2\pi}].\]
By Lemma 3, we can decompose a $(1,1)$ form $\varphi$, by $\varphi=\frac{1}{2}(\Lambda\varphi)\omega +\varphi_{0}$, 
where $\Lambda\varphi=\langle\varphi, \omega\rangle$ and $\varphi_{0}\in\wedge^{1,1}_{0}$.
From this, we can write the Ricci from $\rho$ as $\rho=\frac{1}{4}s\omega+\rho_{0}$.
Using the fact $[\rho]=2\pi c_{1}$ and $d\mu=\frac{[\omega]^{2}}{2}$, we get
\[2\pi c_{1}\cdot[\omega]=[\rho]\cdot[\omega]=\frac{1}{4}[s\omega]\cdot[\omega]=\int_{M}\frac{1}{2}s\ d\mu.\]

\begin{Proposition}
Let $(M, \omega, g)$ be a K\"ahler surface with the scalar curvature $s$.
Then the following identity holds,
\begin{equation}
\int_{M}sd\mu=4\pi c_{1}\cdot[\omega].
\end{equation}
\end{Proposition}
The theorem below tells us which complex surfaces can admit a scalar-flat K\"ahler metric. 
\begin{Theorem}
(Yau) [28], [17] Let $(M,J)$ be a complex surface which admits a Kahler metric $\omega$ such that $c_{1}\cdot[\omega]=0$.
If $c_{1}\neq0$, that is, if $M$ is not covered by a complex torus or K3 surface, then $M$ is a ruled surface. 
\end{Theorem}

\newtheorem{Remark}{Remark}
\begin{Remark}
For a scalar-flat K\"ahler metric, we have $c_{1}\cdot [\omega]=0$ by Proposition 3. Therefore,
by Yau's theorem, a scalar-flat K\"ahler surface is either covered by $K3$ or $T^{4}$ or it is a ruled surface.
Suppose $c_{1}\neq 0$. Since $\omega$ is a self-dual harmonic 2-form, we can conclude $c_{1}$ belongs to $H^{-}$. Thus, $c_{1}^{2}<0.$
Let us consider $\mathbb{CP}^{2}\#n\overline{\mathbb{CP}^{2}}$.
Then we have $\chi=3+n$ and $\tau=1-n$. 
Thus, we have 
\[c_{1}^{2}=2\chi+3\tau=2(3+n)+3(1-n)=9-n.\]
Thus, $\mathbb{CP}^{2}\#n\overline{\mathbb{CP}^{2}}$ with $n\leq 9$ does not admit scalar-flat K\"ahler metrics. 
\end{Remark}
For the existence part, LeBrun constructed explicit scalar-flat K\"ahler metrics on a ruled surface and its blown up,  
$S^{2}\times \Sigma_{\mathbf{g}}\#n\overline{\mathbb{CP}^{2}}$ for $n\geq 2$ and $\mathbf{g}\geq 2$ [13]
and Kim, Pontercorvo extended this result for $n\geq 1$ [11].
Kim, LeBrun and Pontecorvo constructed scalar-flat K\"ahler metrics on $\mathbb{CP}^{2}\#14\overline{\mathbb{CP}^{2}}$ and $S^{2}\times T^{2}\#n\overline{\mathbb{CP}^{2}}$
 for $n=6$ [10] using the result of Donaldson-Friedman [6] .
And Rollin and Singer improved this result on $\mathbb{CP}^{2}\#10\overline{\mathbb{CP}^{2}}$ and $S^{2}\times T^{2}\#n\overline{\mathbb{CP}^{2}}$ for $n=4$ [24].
Note that from the above remark 1, 10 is the minimum number for which scalar-flat K\"ahler metrics can exist on $\mathbb{CP}^{2}\#n\overline{\mathbb{CP}^{2}}$.
Also Kim, Pontercorvo proved one-point blow up of a non-minimal scalar-flat K\"ahler surface also
admits a scalar-flat K\"ahler metric [11]. Therefore, we get the following. 
\begin{Theorem}

[10], [11], [24] Let $M$ be diffeomorphic to one of $\mathbb{CP}^{2}\#n\overline{\mathbb{CP}^{2}}$ for $n \geq 10$.
Then $M$ admits a scalar-flat K\"ahler metric.
\end{Theorem}
We study deformation theory of a scalar-flat K\"ahler metric. 
We deform this metric in two different categories, namely in scalar-flat K\"ahler metrics and ASD metrics.
The obstruction for a deformation as ASD metrics lies in 
\[Coker DW_{+}\cong H^{2}(Z,\Theta_{Z}),\]
where $Z$ is the twistor space of $M$.
For a scalar-flat K\"ahler metric, Pontecorvo's result [23] gives us an additional structure on the twistor space.
Before stating this result, we explain about the twistor space briefly.

An oriented, 4-dimensional riemannian manifold $M$ with an ASD metric has its companion complex 3-dimensional manifold $Z$. 
This $Z$ is the total space of the sphere bundle of the vector bundle of self-dual 2-forms on $M$ and so we have a bundle map $\pi:Z\to M$. 
The Levi-Civita connection on $M$ induces the connection on $TZ$. 
Using this connection, we can split $TZ\cong \mathcal{H}\oplus\mathcal{V}$, where $\mathcal{H}$ is the horizontal part which is isomorphic to $\pi^{*}(TM)$
and $\mathcal{V}$ is the vertical part. Since fibers are 2-spheres, $\mathcal{V}$ has a natural almost-complex structure. 
Given a metric $g$ on $M$, self-dual 2-forms and $g$-orthogonal almost-complex structures on 
$T_{x}M$ correspond via the map
$\omega(v, w)=g(Jv, w)$. Thus, we can think of the fiber over $x$ as the set of all linear maps $J_{x}:TM_{x}\to TM_{x}$ such that $J_{x}^{2}=-1$. 
Suppose $z\in Z$ and $\pi(z)=x$. Then $\mathcal{H}_{z}=\pi^{*}(TM_{x})$ and since $z$ itself represent  a $J_{x}$ in $TM_{x}$, we can assign this $J_{x}$ on $\mathcal{H}_{z}$.
Thus, $TZ$ admits an almost-complex structure. 
The remarkable fact is that when $g$ is ASD, this almost-complex structure on $Z$ is integrable and therefore, the twistor space $Z$ becomes a complex manifold [2]. 
In addition to this, there is a fiberwise antipodal map 
\[\omega \to -\omega\]
and this gives us a fixed-point free anti-holomorphic involution $\sigma$ on the total space $Z$.

Suppose $M$ admits a scalar-flat K\"ahler metric $g$. In this case, we have the following Pontecorvo's result [23]. 
An almost-complex structure $J$ and its conjugate $-J$ give us embeddings of $M$ into $Z$ as complex hypersurfaces and we denote them by $\Sigma$ and 
$\overline{\Sigma}$
and their sum by $D=\Sigma+\overline{\Sigma}$. 
Then the anti-holomorphic involution $\sigma$ interchanges $\Sigma$ and $\overline{\Sigma}$ and so we have 
\[\sigma(\Sigma+\overline{\Sigma})=\overline{\Sigma}+\Sigma.\]
Thus, we get a real bundle $D=[\Sigma+\overline{\Sigma}]$. 
Then the result in [23] gives 
\[[D]\cong K^{-\frac{1}{2}}_{Z}.\]
Conversely, let $Z\to M$ be a twistor fibration and suppose we have a complex hypersurface $\Sigma\subset Z$ which meets every fiber at one point. 
Then $\Sigma$ is diffeomorphic to $M$ and we can think of $M$ as a complex surface induced from $\Sigma$.  
Suppose $[D]\cong K^{-\frac{1}{2}}_{Z}$, where $D=\Sigma+\overline{\Sigma}$. 
Then there is a metric $g$ in the conformal class $[g]$ such that $(M, g, J)$ is scalar-flat K\"ahler. 

Using this result, we have the following, which is originally discovered by Boyer [5] in a different way.

\begin{Theorem}

[5], [23] Let $M$ be an oriented, compact, smooth 4-manifold with an ASD metric $g$ and assume the first Betti number $b_{1}(M)$ is even. 
Suppose there is a complex structure $J$ such that $g(v,w)=g(Jv, Jw)$. 
Then the conformal class of g contains a unique scalar-flat K\"ahler metric. 
\end{Theorem}
From this theorem, we can conclude that deforming scalar-flat K\"ahler metrics is equivalent to deforming ASD hermitian conformal structures. 
The latter correspond to the deformation of the pair $(Z, D)$ preserving the real structure. 
Therefore, when obstruction vanishes, the moduli space of the ASD hermitian conformal structures is the real slice of $H^{1}(Z, \Theta_{Z,D})$.
Details can be found in [17]. From Theorem 5, we can conclude the following. 
\begin{Lemma}
When obstruction vanishes, the deformation of scalar-flat K\"ahler metrics corresponds to the real slice of  $H^{1}(Z,\Theta_{Z,D})$, where $\Theta_{Z,D}$ is the sheaf of holomorphic vector fields
on $Z$ which are tangent to $D$.
And its obstruction lies in $H^{2}(Z,\Theta_{Z,D})$.
\end{Lemma}

\begin{Remark}
Note that deformation of scalar-flat K\"ahler metrics with a fixed complex structure corresponds to the sheaf $\Theta_{Z}\otimes\mathcal{I}_{D}$.
\end{Remark}
Here we denote the ideal sheaf of $D$ by $\mathcal{I}_{D}$. 
LeBrun and Singer identifies $H^{i}(Z,\Theta\otimes\mathcal{I}_{D})$ using the Penrose transform [17].
In particular, if there is no holomorphic vector field or the complex surface is non-minimal, 
it is shown that the obstruction vanishes [10],  [17].  
Below is the result by LeBrun and Singer [17].

\begin{Proposition}

[17] Let $(M,\omega)$ be a compact scalar-flat K\"ahler surface and let $Z$ be its twistor space.
Suppose $c_{1}\neq0$.
If $H^{2}(Z, \Theta\otimes\mathcal{I}_{D})=0$, then $H^{2}(Z, \Theta_{Z,D})=H^{2}(Z,\Theta_{Z})=0$.
\end{Proposition}
This proposition is proved by considering the following short exact sequences of sheaves and showing that $H^{2}(D, \Theta_{D})=0$ and  $H^{2}(D, N_{D})=0$. 
\begin{equation}
0\longrightarrow \Theta_{Z,D}\longrightarrow \Theta_{Z} \longrightarrow N_{D}\longrightarrow 0,
\end{equation}
\begin{equation}
0\longrightarrow \Theta_{Z}\otimes \mathcal{I}_{D}\longrightarrow \Theta_{Z,D} \longrightarrow \Theta_{D}\longrightarrow 0.
\end{equation}
From (6), we get the following long exact sequence, 
\[\cdot\cdot\cdot\longrightarrow H^{2}(Z, \Theta_{Z}\otimes\mathcal{I}_{D})\longrightarrow H^{2}(Z, \Theta_{Z,D})\longrightarrow H^{2}(D, \Theta_{D})\longrightarrow\cdot\cdot\cdot \]
Thus, if $H^{2}(Z, \Theta_{Z}\otimes\mathcal{I}_{D})=0$, then
\[H^{2}(Z, \Theta_{Z,D})=0.\]
Similarly, from (5), we get 
\[\cdot\cdot\cdot\longrightarrow H^{2}(Z, \Theta_{Z,D})\longrightarrow H^{2}(Z, \Theta_{Z})\longrightarrow H^{2}(D, N_{D})\longrightarrow\cdot\cdot\cdot \]
Thus, when $H^{2}(Z, \Theta_{Z,D})=0$, we have $H^{2}(Z, \Theta_{Z})=0$. 

Thus, when $H^{2}(Z, \Theta\otimes\mathcal{I}_{D})=0$,  
obstructions of deformation of scalar-flat K\"ahler metrics and ASD metrics vanish. 
In particular, if there is no holomorphic vector field, then $H^{2}(Z, \Theta_{Z}\otimes\mathcal{I}_{D})=0$ by [17].
When $M$ is a scalar-flat K\"ahler surface but not Ricci-flat, it is shown in [23] 
\[H^{0}(M, \Theta_{M})=H^{0}(Z, \Theta_{Z}).\]
Thus, when there is no holomorphic vector field on $M$, there is no holomorphic vector field on $Z$.
In particular, $H^{0}(Z, \Theta_{Z})=H^{0}(Z, \Theta_{Z, D})=0$. 
Therefore, under this assumption, we get a smooth moduli space of scalar flat K\"ahler metrics and ASD metrics
and their dimensions can be calculated by the index theorem.

\begin{Theorem}
Let $(M, g, J, \omega)$ be a complex surface with a scalar-flat K\"ahler metric. 
Suppose $c_{1}\neq 0$ and assume $(M, J)$ admits no holomorphic vector fields. 
Then on these surfaces, we have $-9\chi\geq 13\tau$ and if the inequality is strict, there exist deformations which are strictly almost-K\"ahler. 
\end{Theorem}
\begin{proof}
Consider the short exact sequence (6). Since $D\cong\Sigma+\overline{\Sigma}$, we have 
\[\chi(Z, \Theta_{Z,D})=\chi(Z, \Theta_{Z}\otimes \mathcal{I}_{D})+2\chi(M, \Theta_{M}).\]
When obstruction vanishes, we can think of $-\chi$ as the dimension of the moduli space.
It is shown in [17] that 
\[\chi(Z, \Theta_{Z}\otimes I_{D})=\tau.\]
Thus, the dimension the moduli space of scalar-flat K\"ahler metrics is given by
\[-\chi(Z, \Theta_{Z,D})=-\tau-2\chi(M, \Theta_{M}).\]
By the index theorem, $\chi(M,\Theta_{M})$ is given by
\[\chi(M,\Theta_{M})=\int_{M}(1+\frac{c_{1}}{2}+\frac{c_{1}^{2}+c_{2}}{12})(2+c_{1}+\frac{c_{1}^{2}-2c_{2}}{2})\]
\[=(\frac{c_{1}^{2}+c_{2}}{6}+\frac{c_{1}^{2}}{2}+\frac{c_{1}^{2}}{2}-c_{2})(M)=(\frac{7}{6}c_{1}^{2}-\frac{5}{6}c_{2})(M)\]
\[=\frac{7(2\chi+3\tau)-5\chi}{6}=\frac{9\chi+21\tau}{6}=\frac{3\chi}{2}+\frac{7\tau}{2}.\]
Therefore, the expected dimension of the moduli space of scalar-flat K\"ahler metrics is 
\[-\chi(Z, \Theta_{Z,D})=-\tau-2\chi(M, \Theta_{M})=-\tau-(3\chi+7\tau)=-3\chi-8\tau.\]

Consider another short exact sequence (5).
By Pontecorvo's result [23], we have $N_{D}=K^{-1}_{D}$.
From this, we get
\[\chi(Z, \Theta_{Z})=\chi(Z, \Theta_{Z, D})+2\chi(M, K^{-1}_{M}).\]
$\chi(M, K^{-1}_{M})$ is given by the Riemann-Roch formula,
\[\chi(L)=\chi(\mathcal{O}_{M})+\int_{M}\frac{c_{1}(L)\left(c_{1}(L)+c_{1}(K^{-1})\right)}{2},\]
where $L$ is any holomorphic line bundle on $M$.
If $L=K^{-1}$, we get
\[\chi(K^{-1})=\chi(\mathcal{O}_{M})+c_{1}(K^{-1})\cdot c_{1}(K^{-1})=\chi(\mathcal{O}_{M})+c_{1}^{2}(M).\]
Since $\chi(\mathcal{O}_{M})=\frac{c_{1}^{2}+c_{2}}{12}$, we have 
\[\chi(K^{-1})=\frac{c_{1}^{2}+c_{2}}{12}+c_{1}^{2}.\]
Using $c_{1}^{2}(M)=2\chi+3\tau$ and $c_{2}(M)=\chi$, we get
\[\chi(M, K^{-1}_{M})=\frac{2\chi+3\tau+\chi}{12}+2\chi+3\tau\]
\[=\frac{1}{4}\chi+\frac{1}{4}\tau+2\chi+3\tau.\]
Then we have 
\[\chi(Z, \Theta_{Z})=\chi(Z, \Theta_{Z, D})+2\chi(M, K^{-1}_{M})\]
\[=3\chi+8\tau+\frac{1}{2}\chi+\frac{1}{2}\tau+4\chi+6\tau\]
\[=\frac{1}{2}(15\chi+29\tau).\]
Thus, when obstruction vanishes, we can conclude the dimension of the moduli space of ASD metrics is given by 
\[-\frac{1}{2}(15\chi+29\tau).\]
Using Lemma 5 below, if 
\[-\frac{1}{2}(15\chi+29\tau)>-(3\chi+8\tau),\]
that is, if $-9\chi>13\tau$, then there exist strictly almost-K\"ahler anti-self-dual metrics. 
In the examples below, we can easily check all scalar-flat K\"ahler surfaces with $c_{1}\neq 0$ satisfy $-9\chi\geq13\tau$.
\end{proof}

\begin{Remark}
The dimension of the moduli of ASD metrics, $-\frac{1}{2}(15\chi+29\tau)$, is given first by I. M. Singer using the Atiyah-Singer Index theorem [8].
Also the dimension of the moduli of scalar-flat K\"ahler metrics, $-(3\chi+8\tau),$ has been known to experts, for example in another version of [11], 
but it seems it has not been written down in detail. 
\end{Remark}

In the following lemma, we show that there is a unique almost-K\"ahler ASD metric in each conformal class which is close to the one containing a scalar-flat K\"ahler metric. 
\begin{Lemma}
Let $M$ be a compact, smooth 4-dimensional manifold which admits a scalar-flat K\"ahler metric.
Assume $M$ does not admit a holomorphic vector field. 
Suppose $g'$ be a small deformation of $g$. 
Then there is a unique almost-K\"ahler anti-self-dual metric in the conformal class of $[g']$.  
\end{Lemma}
\begin{proof}
Since $g'$ is close to the K\"ahler metric, there is a self-dual harmonic 2-form $\omega_{g}'$ is nondegenerate. 
Then by conformal rescaling, we can find a metric $\tilde{g}$ in the conformal class of $g'$ such that $|\omega_{\tilde{g}}|=\sqrt{2}$. 
Then in terms of an orthonormal basis $\{e_{1}, e_{2}, e_{3}, e_{4}\}$, $\omega_{\tilde{g}}=e_{1}\wedge e_{2}+e_{3}\wedge e_{4}$.
Define $Je_{1}=e_{2}, Je_{2}=-e_{1}, Je_{3}=e_{4}, Je_{4}=-e_{3}$. 
Then, $\omega(X, Y)=\tilde{g}(JX, Y)$ and $J$ is $\omega$ and $\tilde{g}$-compatible.

\end{proof}

\begin{Example}
Let us consider $\mathbb{CP}^{2}\#n\overline{\mathbb{CP}^{2}}$.
When there is no holomorphic vector field, 
the dimension the moduli space of scalar-flat K\"ahler metrics is given by 
\[-3\chi-8\tau=-3(3+n)+8(n-1)\]
\[=-9-3n+8n-8=-17+5n.\]
On the other hand, the dimension of the moduli space of ASD metrics is given by 
\[-\frac{(15\chi+29\tau)}{2}=-\frac{\left(15(3+n)+29(1-n)\right)}{2}\]
\[=-\frac{(45+15n+29-29n)}{2}=7n-37.\]
Therefore, we have
\[7n-37-(-17+5n)\geq 0 \iff n\geq 10\]
and equality holds if and only if $n=10$.
This observation with Lemma 5 tells us that for $\mathbb{CP}^{2}\#n\overline{\mathbb{CP}^{2}}$, 
the dimension of the moduli space of almost-K\"ahler ASD metrics is greater than or equal to 
the dimension of the moduli of scalar-flat K\"ahler metrics
if and only if $n\geq 10$ and they are equal when $n=10$.
Thus, when $n>10$, there exists a strictly almost-K\"ahler ASD metric.
\end{Example}
\begin{Example}
Let us consider $S^{2}\times T^{2}\#n\overline{\mathbb{CP}^{2}}$.
Using the fact that $\chi(S^{2}\times T^{2})=\chi(S^{2})\chi(T^{2})=0$ and $\chi(\mathbb{CP}^{2})=3$, 
we get
\[\chi(S^{2}\times T^{2}\#n\overline{\mathbb{CP}^{2}})=\chi(S^{2}\times T^{2})+n(\chi(\overline{\mathbb{CP}^{2}})-2)=n\]
Since $\tau=-n$, the dimension of the moduli space of scalar-flat K\"ahler metrics is given by 
\[-3\chi-8\tau=-3n+8n=5n\]
and the dimension the moduli space of ASD metrics is equal to 
\[-\frac{(15\chi+29\tau)}{2}=-\frac{(15(n)+29(-n))}{2}=7n\]
Thus, when $n\geq 1$, we have a strictly almost-K\"ahler ASD metric.
Note that we only know that scalar-flat K\"ahler metric exist on $S^{2}\times T^{2}\#n\overline{\mathbb{CP}^{2}}$ for $n\geq 4$ [24].

By a similar calculation, we can show that the dimension of the moduli space of almost-K\"ahler ASD metrics is greater than the one of  scalar-flat K\"ahler metrics
in case of $S^{2}\times \Sigma_{\mathbf{g}}\#n\overline{\mathbb{CP}^{2}}$ for $\mathbf{g}\geq 2$. 
And note that LeBrun showed the existence of scalar-flat K\"ahler metric explicitly when $n\geq 2$ [13]
and Kim, Pontecorvo showed the existence of such a metric for $n\geq 1$ [11]. 
\end{Example}

\begin{Example}
Let us consider a minimal ruled surface $\mathbb{V}$ over $\Sigma_{\mathbf{g}}$., where $\mathbf{g}\geq 2$ and $\mathbb{V}$ is a holomorphic rank 2 bundle. 
\end{Example}
When $S^{2}\times \Sigma_{\mathbf{g}}$ admits a standard product K\"ahler metric, there is a holomorphic vector field.
Let us consider projectivization of a rank 2 holomorphic vector bundle $\mathbb{V}$ over $\Sigma_{\mathbf{g}}$.
When $\mathbb{V}$ is a  stable vector bundle over $\Sigma_{\mathbf{g}}$, it can be shown that there is no holomorphic vector field on the ruled surface $P(\mathbb{V})\to \Sigma_{\mathbf{g}}$ [7]. By Narasimhan-Seshadri theorem [21], there is a flat connection on $P(\mathbb{V})$ and therefore, 
universal cover of $P(\mathbb{V})$ is $S^{2}\times \mathcal{H}^{2}$.
Then locally, $P(\mathbb{V})$ is $S^{2}\times\Sigma_{\mathbf{g}}$, where the standard K\"ahler metric is given. 
Therefore, $P(\mathbb{V})$ admits a scalar-flat K\"ahler metric. 

On the other hand, suppose $P(\mathbb{V})$ admits a scalar-flat K\"ahler metric. 
It is shown in [4], [12] that $P(\mathbb{V})$ is locally riemannian product $S^{2}\times\Sigma_{\mathbf{g}}$ with the standard metric.
We briefly discuss the proof. 
Note that in this case, $\tau=0$, and therefore, we have $b_{+}=b_{-}$ and $W_{+}=W_{-}$. 
Thus, there is a self-dual harmonic 2-form $\omega$ and also an anti-self-dual harmonic 2-form $\varphi$. 
From the equation (2), we can conclude $\nabla\omega=0$.
For an anti-self-dual harmonic form $\rho$, we have 
\[0=<\nabla^{*}\nabla\varphi,\varphi>-2W_{-}(\varphi, \varphi)+\frac{s}{3}|\varphi|^{2}.\]
Thus, we can conclude $\nabla\varphi=0$. 
Since there are two parallel 2-forms, the holonomy is a subgroup of $SO(2)\times SO(2)$ and thus we get the conclusion.

In sum, we can conclude scalar-flat K\"ahler metrics on $P(\mathbb{V})$ correspond to the following representation up to conjugation. 
\[\rho: \pi_{1}(M)\to SO(3)\times SO(2,1).\]
Note that $\pi_{1}(M)$ has $2\mathbf{g}$ generators and has 1 relation. Fundamental group $\pi_{1}(M)$ is expressed by
\[\pi_{1}(M)=<a_{1}. b_{1}, ... a_{\mathbf{g}}, b_{\mathbf{g}}| a_{1}b_{1}a_{1}^{-1}b_{1}^{-1}...a_{\mathbf{g}}b_{\mathbf{g}}a_{\mathbf{g}}^{-1}b_{\mathbf{g}}^{-1}=1>.\]
Thus, the dimension of this representation is 
\[6(2\mathbf{g}-2)=12\mathbf{g}-12.\]

We can also count the dimension of the moduli space of scalar-flat K\"ahler metrics from $-3\chi-8\tau$.
When the vector bundle is stable, there is no holomorphic vector field, and thus we can deform the scalar-flat K\"ahler metric on it. 
In this case, $\tau=0$, and therefore, the dimension of the moduli space is $-3\chi$. Since $\chi=-4(\mathbf{g}-1)$, we have
\[-3\chi=3\times 4(\mathbf{g}-1)=12\mathbf{g}-12.\]

The dimension of the moduli space of ASD metrics is given by 
\[-\frac{1}{2}(15\chi+29\tau).\]
Again, since $\tau=0$, it's equal to 
\[-\frac{1}{2}15\chi=\frac{1}{2}15\times 4(\mathbf{g}-1)=30(\mathbf{g}-1).\]

This also can be calculated by considering the corresponding representation. 
ASD metrics on $P(\mathbb{V}),$ over $\Sigma_{\mathbf{g}}$ is conformally flat since $\tau=0$. 
The conformal group acts on the universal covering space, $S^{4}-S^{1}$.
Then the dimension of ASD metric is the same with the dimension of the following representation space up to conjugation. 
\[\rho:\pi_{1}(M)\to SO(5,1).\]
Since $dim SO(5,1)=15$, the dimension of the moduli of ASD metrics which comes from this representation is given by $15(2\mathbf{g}-2)$, which is the same as before. 

\vspace{20pt}

\section{\textbf{Seiberg-Witten invariants}}\label{S:P}

In this section, we show the Seiberg-Witten invariant can give us useful topological information for manifolds which admit a strictly almost-K\"ahler ASD metric and also a metric of positive scalar curvature.

We begin by explaining the Seiberg-Witten invariant briefly.
Suppose a smooth, compact riemannian 4-manifold $(M, g)$ admits an almost-complex structure.
Its homotopy class $c$ contains an almost-complex structure $J$ which is compatible with g [14].
Then the complexified tangent bundle $TM\otimes\mathbb{C}$ decomposes 
as $T^{1,0}\oplus T^{0,1}$. 
We define positive and negative spinor bundles by 
\[\mathbb{V}_{+}=\Lambda^{0,0}\oplus\Lambda^{0,2}\]
\[\mathbb{V}_{-}=\Lambda^{0,1}.\]
These spinor bundles inherits a hermitian inner product from $g$ on $M$. 
Also this bundles have Clifford action of $\Lambda^{p,q}$ given by 
\[v\cdot(w^{1}\wedge\cdot\cdot\cdot\wedge w^{k})=\sqrt{2}v^{0,1}\wedge w^{1}\wedge\cdot\cdot\cdot\wedge w^{k}-\sqrt{2}\sum_{i=1}^{k}\langle w^{i}, \overline{v^{1,0}}\rangle w^{1}\wedge\cdot\cdot\hat{w^{i}}\cdot\cdot\wedge w^{k},\]
where $\langle , \rangle$ is the hermitian inner product which is complex linear on the first variable and anti-linear on the second variable.

These bundles depend only on the homotopy class of J, which we denote by $c=[J]$. This means on a contractible open set $U\in M$, $\mathbb{V}_{\pm}$ can be identified with
 $\mathbb{S}_{\pm}\otimes (K^{-1})^{\frac{1}{2}}$, where $\mathbb{S}_{\pm}$ are spin-bundles and they have spin connections induced from the Levi-Civita connection [14].
 Thus, a connection $A$  on $K^{-1}$ which is compatible with $g$-induced inner product and Spin connection on $\mathbb{S}_{\pm}$
determines a connection on $\mathbb{V}_{\pm}$.
Then using this connection on $\mathbb{V}_{\pm}$, we define the Dirac operator
\[D_{A}:C^{\infty}(\mathbb{V}_{+})\to C^{\infty}(\mathbb{V}_{-}),\]
where $D_{A}$ is given by
\[
\begin{CD}
D_{A}:C^{\infty}(\mathbb{V}_{+})@>\nabla{A}>>\quad C^{\infty}(T^{*}M\otimes\mathbb{V}_{+})@>cl>> C^{\infty}(\mathbb{V}_{-}).\\
\end{CD}
\]
Using an orthonormal basis, it is given by $D_{A}(\Phi)=\Sigma e_{i}\cdot(\nabla_{A}\Phi)(e_{i})$.

The perturbed Seiberg-Witten equation is defined by
\[
\begin{cases}
F_{A}^{+}=\sigma(\Phi)+i\epsilon\\
D_{A}(\Phi)=0.
\end{cases}
\]
Here  $\Phi$ is a section of $C^{\infty}(\mathbb{V}_{+})$ and $F_{A}^{+}$ is the self-dual part of the curvature form of the connection $A$ on $K^{-1}$ and $\epsilon$ is a perturbation self-dual 2-form. 
Since the Lie algebra of $U(1)$ is $i\mathbb{R}$, $F^{+}_{A}\in\Lambda^{2}\otimes i\mathbb{R}$ is a purely imaginary self-dual 2-form.

The Seiberg-Witten equation is invariant under the action of the gauge group, $Map(M, S^{1})$. 
In order to get a well-defined invariant, we need to consider the gauge group action.
We call $(A, \Phi)$ a reducible solution if $\Phi\equiv 0$. 
Otherwise, we call $(A, \Phi)$ an irreducible solution. 
The gauge group does not act freely on reducible solutions, and therefore, we only consider irreducible solutions. 

Let us fix a unitary connection $A_{0}$ on $K^{-1}$. 
Then for any given unitary connection $A$ on $K^{-1}$, there is a gauge transformation so that after the gauge transformation on $A$, 
we have $A=A_{0}+\theta$ and $d^{*}\theta=0$. 
Thus, the moduli space 
\begin{equation}
\mathcal{M}^{*}_{g}=\{(A, \Phi)\in L^{p}_{k}(\Lambda^{1})\times L^{p}_{k}(\mathbb{V}^{+})| D_{A}\Phi=0, F^{+}_{A}=\sigma(\Phi)+i\epsilon, \Phi\neq 0\}/Map(M, S^{1})
\end{equation}
can be rewritten as follows,
\begin{equation}
\mathcal{M}^{*}_{g}=\{(A, \Phi)|D_{A}\Phi=0, F^{+}_{A}=\sigma(\Phi)+i\epsilon, d^{*}(A-A_{0})=0, \Phi\neq 0\}/U^{1}\rtimes H^{1}(M, \mathbb{Z}),
\end{equation}
where $U^{1}\rtimes H^{1}(M, \mathbb{Z})$ is a 1-dimensional group.
In order to define a well-defined map, we choose the space $L^{p}_{k}$, where $p>4$ and $k\geq 1$. 
Here $L^{p}_{k}$ denote the Sobolev space
\[L^{p}_{k}(\Omega)=\{u\in L^{p}(\Omega)|D^{\alpha}u\in L^{p}(\Omega), \forall |\alpha|\leq k\}.\]
Then for generic $\epsilon$, $\mathcal{M}^{*}_{g}$ is a smooth manifold of dimension 0.
We refer to [20] for proofs in detail.

Using the gauge-fixing Lemma, and the bound of a spinor field $\Phi$ which comes from the Seiberg-Witten equation intrinsically, 
we get compactness of $\mathcal{M}^{*}$
and we recover regularity from the elliptic bootstrapping. 
Since the dimension of the moduli space is zero and the moduli space is compact, it consists of finite points.

If $(A,\Phi)$ is reducible, then $F_{A}^{+}=i\epsilon$. Denote by $c_{1}^{+}$ the image of $c_{1}\in H^{2}(M, \mathbb{R})$ in $H^{+}(g)$ of 
deRham classes of self-dual harmonic 2-forms which depends on the metric $g$. 
If the orthogonal projection of $\epsilon$ onto the self dual harmonic 2-forms is not equal to $-2\pi c_{1}^{+}$, then there is no such a solution. However, this is a closed condition, and therefore,
for a generic 2-form, there is no reducible solution. 
\newtheorem{Definition}{Definition}
\begin{Definition}

[14] Let $g$ be a riemannian metric on a compact, smooth 4-dimensional manifold $M$ which admits an almost-complex structure $J$ and let  $\epsilon\in L^{p}_{k-1}(\Lambda^{+})$.
If $[\epsilon_{H}]\neq -2\pi c_{1}^{+}$, then $(g,\epsilon)$ is called a good pair. Here, $\epsilon_{H}$ is a harmonic part of $\epsilon$ and $c_{1}^{+}$ means its projection 
onto self-dual part. 
\end{Definition}
The path component of good pairs is called a chamber. 
If $b_{+}>1$, there is only one chamber and if $b_{+}=1$, there are exactly two chambers [14]. 

\begin{Definition}

[14] We call $(g,\epsilon)$ an excellent pair if $(g,\epsilon)$ is a good pair and if $i\epsilon$ is a regular value for the map $\wp$, where $\wp: \mathcal{M}^{*}_{g}\to F_{A}^{+}-\sigma(\Phi)$. 
\end{Definition}

\begin{Definition}
Let  $(M, c)$ be a smooth, compact 4-manifold with a homotopy class $c=[J]$ of almost-complex structures. 
Then we define by $n_{c}(M, g, \epsilon)$ for an excellent pair $(g, \epsilon)$ the number of solutions of (7) mod 2. 
\end{Definition}

As proven in [14], if two excellent pairs $(g', \epsilon')$ and $(g, \epsilon)$ are in the same chamber, then 
\[n_{c}(M, g',\epsilon')=n_{c}(M, g, \epsilon).\]
 
Thus, when $b_{+}>1$,
we define
\[n_{c}(M)=n_{c}(M, g, \epsilon)\]
for any excellent pair $(g, \epsilon)$. 

When $b_{+}=1$, the Seiberg-Witten invariant depends on the chamber. 
However, if $c_{1}\neq 0$, and $c_{1}^{2}\geq 0$, then $(g,0)$ is a good pair for any metric $g$ and therefore, all the pairs $(g,0)$ belongs to the same path component [14].

We show that there is a non-trivial solution of the Seiberg-Wittten equation for the pair $(g,\epsilon)$, where
$g$ is an almost-K\"ahler anti-self-dual metric and $\epsilon$ is an explicit perturbation form.
Using an almost-K\"ahler metric, Taubes proved there is only one solution for a large perturbation form [27].
And this unique solution consists of special connection $A_{0}$ on the anti canonical bundle $K^{-1}$,
which was discovered independently by Blair and Taubes, and a simple positive spinor $\Phi_{0}=(r,0)\in \mathbb{V}^{+}$.
LeBrun found the curvature form of this connection [16]. Below we only need the self-dual part of $iF_{A_{0}}$.
\begin{Proposition}

[16] Let $M$ be an almost-K\"ahler 4-manifold.  Then the self-dual part of the curvature form of the Blair-Taubes connection $A_{0}$
is given by 
\[iF^{+}_{A_{0}}=\eta+\frac{s+s^{*}}{8}\omega,\]
where $s$ is the scalar curvature and $s^{*}$ is the star-scalar curvature and $\eta=W^{+}(\omega)^{\perp}\in \Lambda^{2,0}\oplus\Lambda^{0,2}$ 
is orthogonal to $\omega$.
\end{Proposition}
When $W^{+}=0$, we have 
\[iF_{A_{0}}^{+}=\frac{s+s^{*}}{8}\omega.\]
By definition, 
\[s^{*}=2R(\omega, \omega).\]
Since $\omega \in \Lambda^{+}$, from the decomposition of the curvature operator, we get

\[R(\omega, \omega)=W^{+}(\omega, \omega)+\frac{s}{12}|\omega|^{2}.\]
Thus, for an ASD metric, we get
\[s^{*}=2R(\omega,\omega)=\frac{s}{3}.\]
In sum, for an  almost-K\"ahler ASD metric, we have
\begin{equation}
iF^{+}_{A_{0}}=\frac{s+s^{*}}{8}\omega=\frac{s+\frac{s}{3}}{8}\omega=\frac{s}{6}\omega.
\end{equation}
In the K\"ahler case, we saw that the following identity (4) holds
\[\int_{M}sd\mu=4\pi c_{1}\cdot[\omega].\]
In the symplectic case, there is a corresponding formula which is discovered by Blair.
In this paper, we can prove this identity using the curvature formula of the Blair-Taubes connection.
\begin{Lemma}

[16] Let $(M, g,\omega)$ be a compact, 4-dimensional almost-K\"ahler manifold. Then, 
\[\int_{M}\frac{s+s^{*}}{2}d\mu=4\pi c_{1}\cdot[\omega].\]
\end{Lemma}
\begin{proof}
By Proposition 5, we have 
\[iF^{+}_{A_{0}}=\eta+\frac{s+s^{*}}{8}\omega\]
Since
\[iF^{+}_{A_{0}}=2\pi c_{1}^{+}(K^{-1}),\]
 we get
\[4\pi c_{1}\cdot\omega=2iF_{A_{0}}\wedge \omega=\int_{M}\frac{s+s^{*}}{4}\omega \wedge \omega=\int_{M}\frac{s+s^{*}}{2}d\mu.\]
\end{proof}
If the metric is almost-K\"ahler anti-self-dual, then $s^{*}$ is equal to $\frac{s}{3}$. Therefore, we get 
\begin{equation}
4\pi c_{1}\cdot[\omega]=\int_{M}\frac{s+s^{*}}{2}d\mu=\int_{M}\frac{s+\frac{s}{3}}{2}d\mu=\int_{M}\frac{2}{3}sd\mu.
\end{equation}
On the other hand, by the Weitzenb\"ock formula (3) for a self dual 2-form, we have 
\[0=|\nabla\omega|^{2}+\frac{2}{3}s.\]
Thus, the scalar curvature $s$ is non-positive and moreover, $s=0$ if and only if $\nabla\omega=0$.
Therefore, for a strictly almost-K\"ahler anti-self-dual metric, we have
\[4\pi c_{1}\cdot[\omega]=\int_{M}\frac{2}{3}sd\mu<0.\]
\begin{Lemma}
Let $g$ be an almost-K\"ahler ASD metric on a compact 4-manifold. 
If $g$ is a strictly almost K\"ahler ASD metric, then $c_{1}\cdot\omega<0$.
\end{Lemma}

\begin{Corollary}
Let $M$ be a smooth, compact, 4-manifold with an almost-K\"ahler anti-self-dual metric $g$. 
Suppose $c_{1}=0$. Then $g$ is hyperK\"ahler. 
\end{Corollary}
\begin{proof}
For an almost-K\"ahler anti-self-dual metric, we have $s\leq 0$.
From (10), if $c_{1}=0$, then $s=0$.
This implies that $g$ is K\"ahler. 
On the other hand,  we have 
\[c_{1}^{2}=2\chi+3\tau=\frac{1}{4\pi^{2}}\int_{M}\left(\frac{|s|^{2}}{24}+2|W^{+}|^{2}-|ric_{0}|^{2}\right)d\mu.\]
Since $g$ is scalar-flat anti-self-dual, we get $ric_{0}=0$.
Thus, $g$ is Ricci-flat K\"ahler. 
\end{proof}

Using an almost-K\"ahler metric, Taubes showed that the constant section $u_{0}$ of $\Lambda^{0,0}$ with unit length satisfies the Dirac equation.
The connection $A_{0}$ on $K^{-1}$ induces a covariant derivative $\nabla_{A_{0}}$ on $\mathbb{V}_{+}$
and $\nabla_{A_{0}}u_{0}\in \mathbb{V}_{+}\otimes T^{*}_{\mathbb{C}}$.
As it is shown by Taubes [27], $A_{0}$ can be chosen so that the following holds
\[\nabla_{A_{0}}u_{0}|_{\Lambda^{0,0}}=0.\]

Then by [27] we note that
$d\omega=0$ if and only if $D_{A_{0}}u_{0}=0$, where $D_{A_{0}}$ is the Dirac operator on $\mathbb{V}_{+}$.

Given $\Phi\in\mathbb{V}_{+}$, we define $\sigma(\Phi)\in End(\mathbb{V_{+}})$ [22] by 
\[\sigma(\Phi): \rho\to \langle\rho, \Phi\rangle\Phi-\frac{1}{2}|\Phi|^{2}\rho.\]

On the other hand, $\Lambda_{+}^{2}\otimes\mathbb{C}$ induces an endomorphism of $\mathbb{V}_{+}$ by Clifford multiplication. 
Let us write $\Phi=(\alpha, \beta)$, where $\alpha\in\Lambda^{0,0}$ and $\beta\in \Lambda^{0,2}$. 
Then we claim the following self-dual 2-form induces $\sigma(\Phi)$,
\[\frac{i}{4}(|\alpha|^{2}-|\beta|^{2})\omega+\frac{1}{2}(\bar{\alpha}\beta-\alpha\bar{\beta}).\]
This can be checked directly. Here we use that $e_{i}\wedge e_{j}$, $i<j$ as an orthonormal basis and $|dz_{i}|=|\overline{dz_{i}}|=\sqrt{2}$. 
Also note that the above 2-form is a purely imaginary self-dual 2-form.
Thus, the Seiberg-Witten equation can be written as
\[
\begin{cases}
F^{+}_{A}=\frac{i}{4}(|\alpha|^{2}-|\beta|^{2})\omega+\frac{1}{2}(\bar{\alpha}\beta-\alpha\bar{\beta})+i\epsilon\\
D_{A}\Phi=0.
\end{cases}
\]
Recall we have the following self-dual part of the curvature formula (9) for $A_{0}$
\[iF^{+}_{A_{0}}=\frac{s+s^{*}}{8}\omega=\frac{s+\frac{s}{3}}{8}\omega=\frac{s}{6}\omega.\]
For the spinor solution $\Phi_{0}=(r,0)$, the corresponding self-dual 2-form is given by
\[\sigma(\Phi_{0})=\frac{ir^{2}\omega}{4}.\]
Therefore, the $\epsilon$ corresponding to $(A_{0}, \Phi_{0})$ on the Seiberg-Witten equation becomes 
\[\epsilon=-(\frac{s}{6}+\frac{r^{2}}{4})\omega.\]
In the following, we show that the Blair-Taubes connection $A_{0}$ and the positive spinor $\Phi_{0}=(r,0)$ is a unique solution 
with respect to the almost-K\"ahler ASD metric $g$ and  $\epsilon$.

\begin{Theorem}
Let $M$ be a smooth, compact 4-manifold with an almost-K\"ahler ASD metric $g$.
Then there is a unique non-trivial solution of the Seiberg-Witten equation for the pair $(g, \epsilon)$, where 
\[\epsilon=-(\frac{s}{6}+\frac{r^{2}}{4})\omega\]
and $r\geq\sqrt{\frac{4|s|}{3}}.$
\end{Theorem}
\begin{proof}
Let us consider following perturbed Seiberg-Witten equation
\[
\begin{cases}
D_{A}\Phi=0\\
F_{A}^{+}=\sigma(\Phi)+i\epsilon.
\end{cases}
\]
A section $\Phi\in \mathbb{V}_{+}=\Lambda^{0,0}\oplus\Lambda^{0,2}$ can be written as
\[\Phi=(\alpha, \beta).\]
The Weitzenb\"ock formula for $D^{*}_{A}D_{A}$ is given by
\[D^{*}_{A}D_{A}\Phi=\nabla_{A}^{*}\nabla_{A}\Phi+\frac{s}{4}\Phi+\frac{F_{A}}{2}\cdot\Phi,\]
where $\cdot$ is Clifford multiplication.
Since $\Phi\in \mathbb{V}^{+}$, we have $\frac{F_{A}}{2}\cdot\Phi=\frac{F_{A}^{+}}{2}\cdot\Phi$.
From the Seiberg Witten equation, we get
\begin{equation}
0=\langle\nabla_{A}^{*}\nabla_{A}\Phi,\Phi\rangle+\frac{s}{4}|\Phi|^{2}+\frac{|\Phi|^{4}}{4}-(\frac{s}{6}+\frac{r^{2}}{4})(|\alpha|^{2}-|\beta|^{2}).
\end{equation}
If we put $\phi=\sigma(\Phi)$, then
$\phi$ is a self-dual 2-form. Moreover,  $\phi$ and $\Phi$ are related in the following way [15],
\[|\phi|^{2}=\frac{1}{8}|\Phi|^{4},\]
\[|\nabla\phi|^{2}\leq \frac{1}{2}|\Phi|^{2}|\nabla\Phi|^{2}.\]
Using the Weitzenb\"ock formula for self-dual 2-form (2), we have 
\[\int_{M}\left(|\nabla\phi|^{2}+\frac{s}{3}|\phi|^{2}\right)d\mu\geq0\]
In terms of $\Phi$, this means
\[\int_{M}\left(|\Phi|^{2}|\nabla\Phi|^{2}+\frac{s}{12}|\Phi|^{4}\right)d\mu\geq0\]
By multiplying $|\Phi|^{2}$ in (11) and using $\langle\nabla_{A}^{*}\nabla_{A}\Phi,\Phi\rangle=|\nabla_{A}\Phi|^{2}+\frac{1}{2}\Delta|\Phi|^{2}$ and
 $\int_{M}|\Phi|^{2}\Delta|\Phi|^{2}d\mu\geq0$, 
it follows 
 \[\int_{M}|\Phi|^{2}\left(\frac{s}{6}|\Phi|^{2}+\frac{|\Phi|^{4}}{4}-(\frac{s}{6}+\frac{r^{2}}{4})(|\alpha|^{2}-|\beta|^{2})\right)d\mu\leq0.\]
By expanding terms and subtracting some positive terms, we have
\[\int_{M}|\Phi|^{2}\left(\frac{s}{3}|\beta|^{2}+\frac{r^{2}(|\alpha|^{2}-r^{2})}{4}+\frac{r^{2}}{4}|\beta|^{2}\right)d\mu\leq0.\]
If we choose $r$ so that 
\[r\geq\sqrt{\frac{4|s|}{3}},\]
then we get 
\[\int_{M}|\Phi|^{2}|\alpha|^{2}d\mu\leq r^{2}\int_{M}|\Phi|^{2}d\mu.\]
Using the fact $|\Phi|^{2}$=$|\alpha|^{2}+|\beta|^{2}$, we have
\[\int_{M}|\alpha|^{2}|\alpha|^{2}d\mu\leq\int_{M}|\Phi|^{2}|\alpha|^{2}d\mu\leq r^{2}\int_{M}\left(|\alpha|^{2}+|\beta|^{2}\right)d\mu.\]
Thus, we get
\[\int_{M}\left(|\alpha|^{2}|\alpha|^{2}-r^{2}|\alpha|^{2}\right)d\mu\leq r^{2}\int_{M}|\beta|^{2}d\mu.\]
Using $F_{A}^{+}=\sigma(\Phi)+i\epsilon$ and $\langle\sigma(\Phi),\omega\rangle=i(\frac{|\alpha|^{2}-|\beta|^{2}}{2})$, we have 
\[|\alpha|^{2}-|\beta|^{2}=-2i\langle F_{A},\omega\rangle+4(\frac{s}{6}+\frac{r^{2}}{4}).\]

Since $c_{1}$ is represented by $[\frac{iF_{A}}{2\pi}]$ for any connection $A$, 
we have $iF_{A}=iF_{A_{0}}+d\gamma$ for the Blair-Taubes connection $A_{0}$ and for some 1-form $\gamma$. 
By (9), we get 
\[|\alpha|^{2}-|\beta|^{2}=-\frac{4s}{6}-2\langle d\gamma,\omega\rangle+4(\frac{s}{6}+\frac{r^{2}}{4}).\]
Thus, we have
\[|\alpha|^{2}-|\beta|^{2}-r^{2}=-2\langle d\gamma,\omega\rangle.\]
Since $\omega$ is a self-dual harmonic 2-form, by integrating, we have
\[\int_{M}\left(|\alpha|^{2}-r^{2}\right)d\mu=\int_{M}\left(|\beta|^{2}-2\langle\gamma,d^{*}\omega\rangle\right)d\mu=\int_{M}|\beta|^{2}d\mu.\]
Using this equality and $\int_{M}\left(|\alpha|^{2}|\alpha|^{2}-r^{2}|\alpha|^{2}\right)d\mu\leq r^{2}\int_{M}|\beta|^{2}d\mu$, we have 
\[\int_{M}\left(|\alpha|^{4}-r^{2}|\alpha|^{2}\right)d\mu\leq r^{2}\int_{M}\left(|\alpha|^{2}-r^{2}\right)d\mu.\]
This implies
\[\int_{M}(|\alpha|^{2}-r^{2})^{2}d\mu=\int_{M}\left(|\alpha|^{4}-r^{2}|\alpha|^{2}-r^{2}|\alpha|^{2}+r^{4}\right)d\mu\leq 0.\]
Thus $|\alpha|^{2}=r^{2}$ and from the equality $\int_{M}\left(|\alpha|^{2}-r^{2}\right)d\mu=\int_{M}|\beta|^{2}d\mu$, it follows  
that $\beta=0$.
Then up to gauge equivalence $\alpha=r$. 
Since the Dirac equation is invariant under the gauge transformation, we have 
\[D_{A}\Phi_{0}=0,\]
where $\Phi_{0}=(r,0)$.
Since $D_{A}\Phi_{0}=D_{A_{0}}\Phi_{0}+\frac{1}{2}\theta\cdot\Phi_{0}$, where $\cdot$ is Clifford multiplication, we get
\[\theta\cdot\Phi_{0}=0.\]
This implies that $\theta^{0,1}=0$.
On the other hand, since $\theta$ is a purely-imaginary 1-form,  we can conclude $\theta=0$. 
Thus, up to gauge transformation, we get the standard solution $(A_{0}, \Phi_{0})$. 
\end{proof}

\begin{Lemma}
Let $g$ be an almost-K\"ahler ASD metric and let $\epsilon$ be given by 
\[\epsilon=-(\frac{s}{6}+\frac{r^{2}}{4})\omega.\]
Then, $(g,\epsilon)$ is an excellent pair.
\end{Lemma}
\begin{proof}
By theorem 7, there is a unique solution $(A_{0}, \Phi_{0})$ up to gauge equivalence. 
We show that this solution is nondegenerate.
This is equivalent to showing the map $\wp$ is surjective, where $\wp:\mathcal{M}^{*}_{g}\to F^{+}-\sigma(\Phi)$.
Let us consider the linearization of the following equations at $\left((A_{0}, \Phi_{0}), -(\frac{s}{6}+\frac{r^{2}}{4})\omega\right),$
\begin{equation}
d^{*}(A-A_{0})=0, D_{A}\Phi=0, F_{A}^{+}=\sigma(\Phi)+i\epsilon.
\end{equation}
Considering (8), the linearization $D\wp$ is onto if and only if the dimension of the kernel of (12) is 1.
Suppose $\left(\theta, (u,\psi)\right)$ belong to the kernel of $(12)$.
We saw that any solution $\Phi=(\alpha, \beta)$ of the equations $D_{A}\Phi=0$ and $F_{A}^{+}=\sigma({\Phi})+i\epsilon$ must be $|\alpha|^{2}=r^{2}, \beta=0$.  
Thus $(u,\psi)$ satisfies the linearization of the pair of the equations $|\alpha|^{2}=r^{2}$ and $\beta=0$. 
The linearization of these equations evaluated at $\Phi_{0}=(r, 0)$ are
\[r(u+\bar u)=0, \quad \psi=0.\]
Note that $u+\bar{u}=0$ implies that the real part of $u$ is zero. Thus, $u$ is purely imaginary. 
Also note that from the linearization of the Dirac equation, we get 
\[D_{A_{0}}(u)=-\frac{1}{2}\theta\cdot(r,0).\]
Since $D_{A_{0}}(u)=\sqrt{2}\bar{\partial}u$, we get
\[\bar{\partial}u=C\theta^{0,1},\]
where $C$ is an explicit constant. 
This implies that $du=\theta$. Since $d^{*}\theta=0$, we get $d^{*}du=0$. By taking an $L^{2}$ inner product with $u$, 
we can conclude $du=0$, and therefore, $u$ is constant and $\theta=0$. 
Since $u$ is purely imaginary and constant, we get the dimension of the kernel of $(12)$ is equal to 1. 
\end{proof}
\begin{Lemma}
Let $M$ be a smooth, compact 4-manifold with $b_{+}=1$ and assume $M$ admits a strictly almost-K\"ahler anti-self-dual metric $g$.
Then $(g,0)$ and $(g,\epsilon)$, where $\epsilon=-(\frac{s}{6}+\frac{r^{2}}{4})\omega$ are good pairs respectively and they can be path connected through good pairs, and therefore, they belong to the same path component. 
\end{Lemma}
\begin{proof}
We use the Blair-Taubes connection $A_{0}$ in order to get $2\pi c_{1}^{+}$.
Since $b_{+}=1$, we can think of $\frac{\omega}{\sqrt{2}}$ as a basis for $\mathcal{H}^{+}$.
Thus, the harmonic part of $iF_{A_{0}}^{+}$ is 
\[\left(\int_{M}\langle\frac{s\omega}{6}, \frac{\omega}{\sqrt{2}}\rangle d\mu\right)\frac{\omega}{\sqrt{2}}=\frac{s_{0}}{6}\omega,\]
where $s_{0}=\int s d\mu$. Since $g$ is a strictly almost-K\"ahler ASD metric, we have $s_{0}<0$.
In particular, this means $2\pi c_{1}^{+}\neq 0$. Thus, $(g,0)$ is a good pair. 
We claim, for $t\in(0,1)$, $(g, t\epsilon)$ is a good pair. 
Suppose not. Then there is $t_{0}\in[0,1]$ such that 
\[\frac{s_{0}}{6}\omega=t_{0}(\frac{s_{0}}{6}+\frac{r^{2}}{4})\omega.\]
Then we rewrite this as
\[(1-t_{0})\frac{s_{0}}{6}\omega=\frac{t_{0}r^{2}}{4}\omega.\]
Since $s_{0}<0$, we get a contradiction.  
\end{proof}

From this, it follows that
$(g,0)$ and $(g,\epsilon)$ belong to the same path component. Therefore, for the chamber which contains $(g,0)$, where $g$ is a strictly almost-K\"ahler ASD metric,
the SW invariant is non-zero.

\begin{Lemma}

[14]  Let $M$ be a compact, 4-dimensional manifold. 
Suppose $M$ admit an almost-complex structure and $b_{+}>0$ and assume $M$ admits a positive scalar curvature metric $\tilde{g}$. 
Then for the chamber whose closure contains $(\tilde{g}, 0)$, $n_{c}=0$. 
\end{Lemma}

Liu's theorem [18] tells us about symplectic manifolds which admits a positive scalar curvature metric. 
\begin{Theorem}
(Liu) Let $M$ be a smooth, compact, symplectic 4-manifold.
If $M$ admits a positive scalar curvature metric, then $M$ is diffeomorphic to either a rational or ruled surfaces. 
\end{Theorem}

\begin{Theorem}
Suppose a smooth, compact 4-manifold $M$ admits an almost K\"ahler ASD metric. 
If $M$ also admits a metric of positive scalar curvature, then it is diffeomorphic to one of the following:

\begin{itemize}
\item$\mathbb{CP}^{2}\#n\overline{\mathbb{CP}^{2}}$  for $n\geq 10$;
\item$S^{2}\times \Sigma_{\mathbf{g}}$ and non-trivial $S^{2}$-bundle over $\Sigma_{\mathbf{g}}$, where $\Sigma_{\mathbf{g}}$
 is a Riemann surface with genus $\mathbf{g}\geq 2$;
\item$(S^{2}\times \Sigma_{\mathbf{g}})\#n\overline{\mathbb{CP}^{2}}$ for $n\geq 1$; or
\item$(S^{2}\times T^{2})\#n\overline{\mathbb{CP}^{2}}$ for $n\geq 1$.
 \end{itemize}

\end{Theorem}
\begin{proof}

By Liu's theorem, $M$ is diffeomorphic to a rational or ruled surfaces. 
Suppose $g$ be a strictly almost-K\"ahler ASD metric. 
Then the scalar curvature $s$ is negative somewhere. Then $c_{1}\neq 0$
and $b_{+}=1$. We show $c_{1}^{2}<0$.
If $c_{1}^{2}\geq 0$, then all pairs $(g,0)$ is a good pair and belong to the same chamber. 
However, by Lemma 10, for the chamber which contains $\tilde{g}$, a metric of positive scalar curvature, $n_{c}=0$. 
On the other hand, for the chamber which contains $(g,0)$, where $g$ is a strictly almost-K\"ahler ASD metric, $n_{c}=1$. 
Thus, we get a contradiction and therefore, $c_{1}^{2}\geq 0.$
Thus, we get the conclusion in this case. 
Note that by Gromov-Lawson [9],  all of these examples admit a metric of positive scalar curvature.
 
Suppose $g$ be a scalar-flat K\"ahler metric. 
Then by Yau's theorem, $M$ has either $c_{1}=0$, or it is a ruled surface. 
Suppose $c_{1}\neq 0$. 
Then, from Remark 1, we get the conclusion in this case. 

Suppose $M$ admits a scalar-flat K\"ahler metric and $c_{1}=0$.
Then, universal cover of $M$ is to either $T^{4}$ or $K3$. 
However, these do not belong to the list of Liu's theorem. 
Thus, $c_{1}\neq0$.
\end{proof}

\begin{Theorem}
Suppose $\mathbb{CP}^{2}\#n\overline{\mathbb{CP}^{2}}$ admits an almost-K\"ahler ASD metric.
Then $n\geq 10$.
\end{Theorem}
\begin{proof}
In this case, $c_{1}\neq 0$. If $n\leq 9$, then $c_{1}^{2}\geq 0$ and therefore, all the pairs $(g,0)$ belong to the same chamber.
Since these manifolds admit a positive scalar curvature metric, it follows that any almost-K\"ahler anti-self-dual metric is scalar-flat K\"ahler.
Then, we have $c_{1}^{2}<0$, which is a contradiction.
Thus,  $\mathbb{CP}^{2}\#n\overline{\mathbb{CP}^{2}}$ for $n\leq 9$
does not admit an almost-K\"ahler anti-self-dual metric. 
\end{proof}

\newpage

\renewcommand{\refname}{Bibliography}

\end{document}